\documentclass{amsart}
\usepackage{amsfonts}
\usepackage{amssymb,latexsym}
\usepackage{amsmath}
\usepackage{amssymb}
\usepackage[neverdecrease]{paralist}

\usepackage[usenames]{color}

\newcommand{\vect}[1]{\ensuremath{\mathbf{#1}}} % vector
\newcommand{\card}[1]{\ensuremath{\lvert{#1}\rvert}} % cardinality or length
\newcommand{\minor}[3]{\ensuremath{{#1}_{{#2} \gets {#3}}}} % variable identification minor
\DeclareMathOperator{\ess}{ess} % essential arity
\DeclareMathOperator{\Ess}{Ess} % the set of indices of essential variables
\DeclareMathOperator{\gap}{gap} % arity gap
\DeclareMathOperator{\qa}{qa} % quasi-arity
\DeclareMathOperator{\GF}{GF} % Galois field
\newcommand{\oddsupp}{\ensuremath{\mathrm{oddsupp}}} % oddsupp
\newcommand{\Aneq}[1][n]{\ensuremath{A^{#1}_{=}}}

\theoremstyle{plain}
\newtheorem{theorem}{Theorem}[section]
\newtheorem{proposition}[theorem]{Proposition}
\newtheorem{lemma}[theorem]{Lemma}
\newtheorem{corollary}[theorem]{Corollary}
\newtheorem{fact}[theorem]{Fact}

\theoremstyle{definition}
\newtheorem{example}[theorem]{Example}
\theoremstyle{remark}
\newtheorem{remark}[theorem]{Remark}

\begin{document}
\title{On the arity gap of polynomial functions}
\author{Miguel Couceiro}
\address[M. Couceiro]{Mathematics Research Unit \\
University of Luxembourg \\
6, rue Richard Coudenhove-Kalergi \\
L-1359 Luxembourg \\
Luxembourg}
\email{miguel.couceiro@uni.lu}

\author{Erkko Lehtonen}
\address[E. Lehtonen]{Computer Science and Communications Research Unit \\
University of Luxembourg \\
6, rue Richard Coudenhove-Kalergi \\
L-1359 Luxembourg \\
Luxembourg}
\email{erkko.lehtonen@uni.lu}

\author{Tam\'as Waldhauser}
\address[T. Waldhauser]{Mathematics Research Unit \\
University of Luxembourg \\
6, rue Richard Coudenhove-Kalergi \\
L-1359 Luxembourg \\
Luxembourg
\and
Bolyai Institute \\
University of Szeged \\
Aradi v\'{e}rtan\'{u}k tere 1 \\
H-6720 Szeged \\
Hungary}
\email{twaldha@math.u-szeged.hu}

\begin{abstract}
The authors' previous results on the arity gap of functions of several variables are refined by considering polynomial functions over arbitrary fields. We explicitly describe the polynomial functions with arity gap at least $3$, as well as the polynomial functions with arity gap equal to $2$ for fields of characteristic $0$ or $2$. These descriptions are given in the form of decomposition schemes of polynomial functions.
Similar descriptions are given for arbitrary finite fields. However, we show that these descriptions do not extend to infinite fields of odd characteristic.
\end{abstract}
\maketitle

\section{Introduction and preliminaries}
\label{sect prel}

Throughout this section, let $A$ and $B$ be arbitrary sets with at least two elements. A \emph{partial function of several variables} from $A$ to $B$ is a mapping $f \colon S \to B$, where $S \subseteq A^n$ for some integer $n \geq 1$, called the \emph{arity} of $f$. If $S = A^n$, then we speak of (\emph{total}) \emph{functions of several variables.} Functions of several variables from $A$ to $A$ are referred to as \emph{operations} on $A$.

For an integer $n \geq 1$, let $[n] := \{1, \dots, n\}$.
Let $f \colon S \to B$ ($S \subseteq A^n$) be an $n$-ary partial function and let $i \in [n]$.
We say that the $i$-th variable is \emph{essential} in $f$ (or $f$ \emph{depends} on $x_i$), if there exist tuples
\[
(a_1, \dots, a_{i-1}, a_i, a_{i+1}, \dots, a_n), (a_1, \dots, a_{i-1}, a'_i, a_{i+1}, \dots, a_n) \in S
\]
such that
\[
f(a_1, \dots, a_{i-1}, a_i, a_{i+1}, \dots, a_n) \neq f(a_1, \dots, a_{i-1}, a'_i, a_{i+1}, \dots, a_n).
\]
Variables that are not essential are called \emph{inessential.}
Let $\Ess f := \{i \in [n] : \text{$x_i$ is essential in $f$}\}$. The cardinality of $\Ess f$ is called the \emph{essential arity} of $f$ and denoted by $\ess f$.

Let $f \colon A^n \to B$, $g \colon A^m \to B$. We say that $g$ is a \emph{simple minor} of $f$, if there is a map $\sigma \colon [n] \to [m]$ such that $g(x_1, \dots, x_m) = f(x_{\sigma(1)}, \dots, x_{\sigma(n)})$. We say that $f$ and $g$ are \emph{equivalent} if each one is a simple minor of the other.

For $i, j \in [n]$, $i \neq j$, define the \emph{identification minor} of $f \colon A^n \to B$ obtained by identifying the $i$-th and the $j$-th variable, as the simple minor $\minor{f}{i}{j} \colon A^n \to B$ of $f$ corresponding to the map $\sigma \colon [n] \to [n]$, $i \mapsto j$, $\ell \mapsto \ell$ for $\ell \neq i$, i.e., $\minor{f}{i}{j}$ is given by the rule
\[
\minor{f}{i}{j}(x_1, \dots, x_n) := f(x_1, \dots, x_{i-1}, x_j, x_{i+1}, \dots, x_n).
\]

\begin{remark}
Loosely speaking, a function $g$ is a simple minor of $f$, if $g$ can be obtained from $f$ by permutation of variables, addition of inessential variables and identification of variables. Similarly, two functions are equivalent, if each one can be obtained from the other by permutation of variables and addition or deletion of inessential variables.
\end{remark}

The \emph{arity gap} of $f$ is defined as
\[
\gap f := \min_{\substack{i,j \in \Ess f \\ i \neq j}} (\ess f - \ess \minor{f}{i}{j}).
\]

\begin{remark}
\label{rem:essgap}
Note that the definition of arity gap refers only to essential variables. Hence, in order to determine the arity gap of a function $f$, we may consider, instead of $f$, an equivalent function $f'$ that is obtained from $f$ by removing its inessential variables. It is easy to see that in this case $\gap f = \gap f'$. Therefore, whenever we consider the arity gap of a function $f$, we may assume without loss of generality that $f$ depends on all of its variables.
\end{remark}

\begin{example}
\label{ex:field}
Let $F$ be an arbitrary field. Consider the polynomial function $f \colon F^3 \to F$ induced by $x_1 x_3 - x_2 x_3$. It is clear that all variables of $f$ are essential, i.e., $\ess f = 3$. Let us form the various identification minors of $f$:
\begin{align*}
& \minor{f}{1}{2}(x_1, x_2, x_3) = 0, &
& \minor{f}{2}{1}(x_1, x_2, x_3) = 0, \\
& \minor{f}{1}{3}(x_1, x_2, x_3) = x_3^2 - x_2 x_3, &
& \minor{f}{3}{1}(x_1, x_2, x_3) = x_1^2 - x_1 x_2, \\
& \minor{f}{2}{3}(x_1, x_2, x_3) = x_1 x_3 - x_3^2, &
& \minor{f}{3}{2}(x_1, x_2, x_3) = x_1 x_2 - x_2^2.
\end{align*}
The essential arities of the identification minors are
\begin{align*}
& \ess \minor{f}{1}{2} = \ess \minor{f}{2}{1} = 0, \\
& \ess \minor{f}{1}{3} = \ess \minor{f}{3}{1} = \ess \minor{f}{2}{3} = \ess \minor{f}{3}{2} = 2.
\end{align*}
We conclude that $\gap f = 1$.
\end{example}

\begin{example}
\label{ex:linearBf}
Let $f \colon \{0,1\}^n \to \{0,1\}$ ($n \geq 2$) be the function induced by the polynomial $x_1 + x_2 + \dots + x_n + c$ ($c \in \{0,1\}$) over the two-element field. Then for each $i \neq j$ we have that $\minor{f}{i}{j}$ is induced by the polynomial
\[
\Bigl( \sum_{\ell \in [n] \setminus \{i, j\}} x_\ell \Bigr) + c.
\]
Thus $\ess f = n$ and $\ess \minor{f}{i}{j} = n - 2$ for all $i \neq j$; hence $\gap f = 2$.
It was shown by Salomaa~\cite{Salomaa} that every operation on $\{0,1\}$ has arity gap at most $2$.
The operations on $\{0,1\}$ were classified according to their arity gap in~\cite{CL2007}, where it was shown that for $n \geq 4$, the linear functions mentioned above are the only operations on $\{0,1\}$ that have essential arity $n$ and arity gap equal to $2$.
\end{example}

\begin{example}
Let $A$ be a finite set with $k \geq 2$ elements, say, $A = \{0, 1, \dots, k-1\}$. Let $f \colon A^n \to A$, $2 \leq n \leq k$, be given by the rule
\[
f(a_1, \dots, a_n) :=
\begin{cases}
1 & \text{if $(a_1, \dots, a_n) = (0, 1, \dots, n-1)$,} \\
0 & \text{otherwise.}
\end{cases}
\]
It is easy to see that all variables of $f$ are essential, and for all $i \neq j$, the function $\minor{f}{i}{j}$ is identically $0$. Hence $\gap f = n$. This example illustrates the fact that there exist functions of arbitrarily high arity gap.
\end{example}

The notion of arity gap has been studied by several authors~\cite{CL2007,CL,CLISMVL2010,GapA,GapB,oddsupp,Salomaa,SK,Willard}. In~\cite{CL}, a general classification of functions according to their arity gap was established. In order to state this result, we need to recall a few notions.

For $n \geq 2$, define
\[
\Aneq := \{(a_1, \dots, a_n) \in A^n : \text{$a_i = a_j$ for some $i \neq j$}\}.
\]
Furthermore, define $\Aneq[1] := A$.
Let $f \colon A^n \to B$. Any function $g \colon A^n \to B$ satisfying $f|_{A^n_=} = g|_{A^n_=}$ is called a \emph{support} of $f$. The \emph{quasi-arity} of $f$, denoted $\qa f$, is defined as the minimum of the essential arities of all supports of $f$, i.e., $\qa f := \min_g \ess g$ where $g$ ranges over the set of all supports of $f$. If $\qa f = m$, then we say that $f$ is \emph{quasi-$m$-ary.} Note that if $A$ is finite and $n > \card{A}$, then $A^n_= = A^n$; hence in this case $\qa f = \ess f$. Moreover, for an arbitrary $A$ and $n \neq 2$, we have $\qa f = \ess f|_{A^n_=}$ (see Lemma~4 in~\cite{CL}). The case $n = 2$ is excluded, because if $f \colon A^2 \to B$ is a function such that $f(a,a) \neq f(b,b)$ for some $a, b \in A$, then $\qa f = 1$ yet $\ess f|_{A^n_=} = 0$.

Denote by  $\mathcal{P}(A)$ the power set of $A$, and define the function $\oddsupp \colon \bigcup_{n \geq 1} A^n \to \mathcal{P}(A)$ by
\[
\oddsupp(a_1, \dots, a_n) :=
\{a \in A : \text{$\card{\{j \in [n] : a_j = a\}}$ is odd}\}.
\]
We say that a partial function $f \colon S \to B$ ($S \subseteq A^n$) is \emph{determined by $\oddsupp$} if there exists a function $f^* \colon \mathcal{P}(A) \to B$ such that
\begin{equation}
f = f^* \circ {\oddsupp}|_S.
\label{eq:fstar}
\end{equation}
Observe that only the restriction of $f^*$ to the set
\[
\mathcal{P}'_n(A) := \bigl\{ T \in \mathcal{P}(A) : \card{T} \in \{n, n-2, n-4, \dots\} \bigr\},
\]
is relevant in determining the values of $f$ in~\eqref{eq:fstar}. Moreover, the functions $f \colon A^n \to B$ determined by $\oddsupp$ are in one-to-one correspondence with the functions $f^* \colon \mathcal{P}'_n(A) \to B$.

The notion of a function being determined by $\oddsupp$ is due to Berman and Kisielewicz~\cite{BK}. Willard showed in~\cite{Willard} that if $f \colon A^n \to B$, where $A$ is finite, $\ess f = n > \max(\card{A}, 3)$ and $\gap f \geq 2$, then $f$ is determined by $\oddsupp$. The following fact is easy to verify.

\begin{fact}
\label{fact oddsupp}
A function $f \colon A^n \to B$ is determined by $\oddsupp$ if and only if $f$ is totally symmetric and $\minor{f}{2}{1}$ does not depend on $x_1$.
Similarly, $f|_{\Aneq}$ is determined by $\oddsupp$ if and only if $f|_{\Aneq}$ is totally symmetric and $\minor{f}{2}{1}$ does not depend on $x_1$.
\end{fact}

We can now state the general classification of functions according to the arity gap. This result was first obtained in~\cite{CL} for functions with finite domains, and in~\cite{GapA} it was shown to still hold for functions with arbitrary, possibly infinite domains.

\begin{theorem}
\label{thm:gap}
Let $A$ and $B$ be arbitrary sets with at least two elements.
Suppose that $f \colon A^n \to B$, $n \geq 2$, depends on all of its variables.

\begin{enumerate}[\rm (i)]
\item \label{thmitem1}
For\/ $3 \leq p \leq n$, $\gap f = p$ if and only if\/ $\qa f = n - p$.

\item \label{thmitem2}
For $n \neq3$, $\gap f = 2$ if and only if
\begin{itemize}
\item $\qa f = n - 2$ or
\item $\qa f = n$ and $f|_{\Aneq}$ is determined by\/ $\oddsupp$.
\end{itemize}

\item \label{thmitem3}
For $n = 3$, $\gap f = 2$ if and only if there is a nonconstant unary function $h \colon A \to B$ and $i_1, i_2, i_3 \in \{0, 1\}$ such that
\begin{align*}
f(x_1, x_0, x_0)  &  = h(x_{i_1}), \\
f(x_0, x_1, x_0)  &  = h(x_{i_2}), \\
f(x_0, x_0, x_1)  &  = h(x_{i_3}).
\end{align*}

\item
Otherwise\/ $\gap f = 1$.
\end{enumerate}
\end{theorem}

Theorem~\ref{thm:gap} can be refined to obtain more explicit classifications by assuming certain structures on the domain $A$ or the codomain $B$ of $f$. Examples of such refinements include the complete classification of Boolean functions~\cite{CL2007}, pseudo-Boolean functions~\cite{CL}, lattice polynomial functions~\cite{CLISMVL2010}, or more generally, order-preserving functions~\cite{GapB}. Moreover, in~\cite{GapA}, $B$ was assumed to be a group, and the following decomposition scheme based on the quasi-arity was obtained.

\begin{theorem}
\label{thm:sum}
Assume that $(B; +)$ is a group with neutral element\/ $0$. Let $f \colon A^n \to B$, $n \geq 3$, and\/ $1 \leq p \leq n$. Then the following two conditions are equivalent:
\begin{enumerate}[\rm (i)]
\item \label{itemlemmasum1}
$\ess f = n$ and\/ $\qa f = n - p$.

\item \label{itemlemmasum2}
There exist functions $g, h \colon A^n \to B$ such that $f = g + h$, $h|_{\Aneq} \equiv 0$, $h \not\equiv 0$, and\/ $\ess g = n - p$.
\end{enumerate}
The decomposition $f = g + h$ given above is unique.
\end{theorem}

Theorems~\ref{thm:gap} and~\ref{thm:sum} lead to the following characterization of functions with arity gap at least $3$. A similar description was proposed by Shtrakov and Koppitz~\cite{SK}.

\begin{corollary}
\label{cor:sum}
Assume that $(B; +)$ is a group with neutral element\/ $0$. Let $f \colon A^n \to B$, $n \geq 3$, and\/ $3 \leq p \leq n$. Then the following two conditions are equivalent:

\begin{enumerate}[\rm (i)]
\item \label{itemthmsum1}
$\ess f = n$ and\/ $\gap f = p$.

\item \label{itemthmsum2}
There exist functions $g, h \colon A^n \to B$ such that $f = g + h$, $h|_{\Aneq} \equiv 0$, $h \not\equiv 0$, and\/ $\ess g = n - p$.
\end{enumerate}
The decomposition $f = g + h$ given above is unique.
\end{corollary}

Analogous decompositions $f = g + h$ were presented in~\cite{GapA} for functions $f \colon A^n \to B$ with $\gap f = 2$, in which either $\ess g = n - 2$ or $g$ is a sum of functions of essential arity at most $n - 2$.

Having the previous results as our starting point, we present in the current paper yet another refinement of Theorem~\ref{thm:gap}. Namely, we study the arity gap of polynomial functions over arbitrary fields. We will obtain further, more explicit decomposition schemes.

The paper is organised as follows. In Section~\ref{sect general fields}, we recall the basic notions and introduce preliminary results which will be needed throughout the paper. In particular, we provide a general decomposition scheme for polynomial functions over arbitrary fields with arity gap at least $3$.
In subsequent sections we focus on functions with arity gap $2$.
More precisely, in Section~\ref{sect char2}, we describe the polynomial functions determined by $\oddsupp$, and we obtain decomposition schemes for functions with arity gap $2$ over finite fields and fields of characteristic $2$.
In Section~\ref{sect char0}, we consider the case of fields of characteristic $0$. In this case, we show that if $f$ is a polynomial function such that $f|_{A^n_{=}}$ is determined by $\oddsupp$, then $f|_{A^n_{=}}$ is constant. Hence, simpler decomposition schemes are available for polynomial functions with arity gap $2$.
The question whether similar decomposition schemes exist over infinite fields of odd characteristic is addressed in Section~\ref{sec:oddchar}. We answer negatively to this question by means of an illustrative example.

%%%%%%%%%%%%%%%%%%%%

\section{Arity gap of polynomial functions over fields}
\label{sect general fields}

In what follows, we will assume that the reader is familiar with the basic notions of algebra, such as rings, unique factorization domains, fields, vector spaces, polynomials and polynomial functions. However, we find it useful to recall the following well-known result.

\begin{fact}
\label{fact:field}
Every function $f \colon F^n \to F$ on a finite field $F$ is a polynomial function over $F$.
\end{fact}

Polynomials over infinite fields are in one-to-one correspondence with polynomial functions.
Fact~\ref{fact:field} establishes a correspondence between polynomials and functions over finite fields, which is not bijective. This correspondence can be made bijective by assuming that we only consider polynomials over a given finite field, say $F = \GF(q)$, in which the exponent of every variable in every monomial is at most $q-1$; we shall call such polynomials over finite fields \emph{canonical.} In the case of infinite fields, every polynomial is \emph{canonical.}

Given a polynomial function $f \colon F^n \to F$, we denote by $P_f$ the unique canonical polynomial which induces $f$. Given a polynomial $p \in F[x_1, \dots, x_n]$, we denote by $\overline{p}$ the function $f \colon F^n \to F$ induced by $p$.
Note that $\overline{p + q} = \overline{p} + \overline{q}$ for all $p, q \in F[x_1, \dots, x_n]$.

\begin{fact}
\label{fact:essential}
A variable $x_i$ is essential in a polynomial function $f \colon F^n \to F$ if and only if $x_i$ occurs in $P_f$.
\end{fact}

Let $F$ be a field, and let us apply the results of Section~\ref{sect prel} in the case $A = B = F$ for polynomial functions $f \colon F^n \to F$.

\begin{lemma}
\label{lemma minors}
If $f$ is a polynomial function over $F$, then the functions $g$ and $h$ in the decomposition $f = g + h$ given in Theorem~\ref{thm:sum} and Corollary~\ref{cor:sum} are also polynomial functions.
\end{lemma}

\begin{proof}
Since $\ess g = n - p \leq n - 1$, the function $g$ has an inessential variable, say the $i$-th variable is inessential in $g$. Let $j \neq i$. We clearly have $\minor{g}{i}{j} = g$, and since $h|_{\Aneq} \equiv 0$, we have
\[
\minor{f}{i}{j} = \minor{g}{i}{j} + \minor{h}{i}{j} = g + 0 = g.
\]
Thus, $g$ is a simple minor of $f$ and hence a polynomial function. Then $h = f - g$ is a polynomial function as well.
\end{proof}

\begin{lemma}
\label{lemma Delta}
If $h$ is an $n$-ary polynomial function over $F$, then $h|_{F^n_{=}} \equiv 0$ if and only if $h$ is induced by a multiple of the polynomial
\[
\Delta_n = \prod_{1 \leq i < j \leq n} (x_i - x_j) \in F[x_1, \dots, x_n].
\]
\end{lemma}

\begin{proof}
It is clear that if $h$ is induced by a multiple of $\Delta_n$, then $h|_{F^n_{=}} \equiv 0$.
For the converse implication, we need to distinguish between the cases of finite and infinite $F$. Assume first that $F$ is infinite, and let us suppose that $h|_{F^n_{=}} \equiv 0$.
Let us consider $P_h$ as an element of $R[x_n]$, where $R$ denotes the ring $F[x_1, \dots, x_{n-1}]$. Since $h|_{F^n_{=}} \equiv 0$, each one of the elements $x_1, \dots, x_{n-1} \in R$ is a root of the unary polynomial $P_h(x_n) \in R[x_n]$.
Therefore $P_h$ is divisible by $x_i - x_n$ for all $i = 1, \dots, n - 1$. Repeating this argument with $x_j$ in place of $x_n$, we can see that $x_i - x_j$ divides $P_h$ for all $1 \leq i < j \leq n$. Since these divisors of $P_h$ are relatively prime (and $F[x_1, \dots, x_n]$ is a unique factorization domain), we can conclude that $P_h$ is divisible by their product $\Delta_n$.

Assume then that $F$ is finite. Define the function $h' \colon F^n \to F$ by the rule
\[
h'(\vect{a}) =
\begin{cases}
h(\vect{a}) \cdot (\overline{\Delta_n}(\vect{a}))^{-1}, & \text{if $\vect{a} \in F^n \setminus F^n_{=}$,} \\
0, & \text{if $\vect{a} \in F^n_{=}$.}
\end{cases}
\]
Observe that $\overline{\Delta_n}(\vect{a}) \neq 0$ for every $\vect{a} \in F^n \setminus F^n_{=}$; hence $h'$ is well defined. (In fact, $h'$ could be defined in an arbitrary way on $F^n_{=}$.) Clearly $h = h' \cdot \overline{\Delta_n}$.
By Fact~\ref{fact:field}, $h'$ is a polynomial function. Thus $h$ is induced by the polynomial $P_{h'} \cdot \Delta_n$.
\end{proof}

Combining the previous two lemmas with Corollary~\ref{cor:sum}, we obtain the following description of polynomial functions over $F$ with arity gap at least $3$.

\begin{theorem}
\label{thm poly gap3}
Let $F$ be a field and let $f \colon F^n \to F$ be a polynomial function of arity at least $4$. Then $\gap f = p \geq 3$ if and only if there exist polynomials $P, Q \in F[x_1, \dots, x_n]$ such that $f = \overline{P} + \overline{Q}$, $P$ is canonical, exactly $n - p$ variables occur in $P$, and $Q$ is a nonzero multiple of the polynomial $\Delta_n$ such that $\overline{Q}$ is not identically $0$. Moreover, if $f = \overline{P'} + \overline{Q'}$, where $P'$ is canonical, $n - p$ variables occur in $P'$ and $Q'$ is a nonzero multiple of $\Delta_n$ such that $\overline{Q'}$ is not identically $0$, then $P' = P$ and $\overline{Q'} = \overline{Q}$.
\end{theorem}

%%%%%%%%%%%%%%%%%%%%

\section{Functions determined by $\oddsupp$ and the arity gap of polynomial functions over fields of characteristic $2$}
\label{sect char2}

We refine Fact~\ref{fact oddsupp} for polynomial functions over an arbitrary field $F$. To this extent, we need some formalism. We use the following notation:
\begin{itemize}
\item If $F$ is infinite, then $N_F$ denotes the set $\mathbb{N}$ of nonnegative integers, $M_F$ denotes the set of all nonnegative even integers, and $\oplus_F$ denotes the usual addition of nonnegative integers.
\item If $F$ has finite order $q$, then $N_F$ denotes the set $\{0, 1, \dots, q-1\}$, $M_F := N_F$, and $\oplus_F$ is the operation on $N_F$ given by the following rules:
\begin{itemize}
\item $0 \oplus_F 0 = 0$.
\item If $a \neq 0$ or $b \neq 0$, then $a \oplus_F b = c$, where $c$ is the unique number in $\{1, \dots, q-1\}$ such that $c \equiv a + b \pmod{q-1}$.
\end{itemize}
\end{itemize}
Define the map $\tau_F \colon N_F \to M_F$ by the rule $m \mapsto m \oplus_F m$.

\begin{remark}
\label{rem:tauF}
If $F$ is infinite or of even order, then $\tau_F$ is a bijection that has $0$ as a fixed point.
\end{remark}

\begin{lemma}
\label{lem:f21}
Let $F$ be an arbitrary field, and let $f \colon F^n \to F$ be a polynomial function with
\[
P_f = \sum_{\vect{k} = (k_1, \dots, k_n) \in N_F^n} c_{\vect{k}} x_1^{k_1} x_2^{k_2} \cdots x_n^{k_n}.
\]
Then $\minor{f}{2}{1}$ does not depend on $x_1$ if and only if for all $(k, k_3, \dots, k_n) \in N_F^{n-1}$ with $k \neq 0$,
\[
\sum_{\substack{(a_1, a_2) \in N_F^2 \\ a_1 \oplus_F a_2 = k}} c_{(a_1, a_2, k_3, \dots, k_n)} = 0.
\]
\end{lemma}

\begin{proof}
The canonical polynomial for $\minor{f}{2}{1}$ is
\[
\sum_{(b_1, b_3, \dots, b_n) \in N_F^{n-1}} d_{(b_1, b_3, \dots, b_n)} x_1^{b_1} x_3^{b_3} \cdots x_n^{b_n},
\]
where
\[
d_{(b_1, b_3, \dots, b_n)} =
\sum_{\substack{(a_1, a_2) \in N_F^2 \\ a_1 \oplus a_2 = b_1}} c_{(a_1, a_2, b_3, \dots, b_n)}.
\]
By Fact~\ref{fact:essential}, the condition that $\minor{f}{2}{1}$ does not depend on $x_1$ is equivalent to the condition that $d_{(b_1, b_3, \dots, b_n)} = 0$ for all $(b_1, b_3, \dots, b_n) \in N_F^{n-1}$ such that $b_1 \neq 0$.
\end{proof}

\begin{proposition}
\label{prop oddsuppchar2}
Let $F$ be an arbitrary field, and let $f \colon F^n \to F$ be a polynomial function with
\[
P_f = \sum_{\vect{k} = (k_1, \dots, k_n) \in N_F^n} c_{\vect{k}} x_1^{k_1} x_2^{k_2} \cdots x_n^{k_n}.
\]

Then $f$ is determined by $\oddsupp$ if and only if
\begin{enumerate}[\rm (A)]
\item \label{cond symmetric}
$f$ is symmetric, i.e.,
$c_{(k_1, \dots, k_n)} = c_{(l_1, \dots, l_n)}$
whenever there is a permutation $\pi \in S_n$ such that $k_i = l_{\pi(i)}$ for all $i \in [n]$, and

\item \label{cond coeff}
for all $(k, k_3, \dots, k_n) \in N_F^{n-1}$ with $k \neq 0$,
\[
\sum_{\substack{(a_1, a_2) \in N_F^2 \\ a_1 \oplus_F a_2 = k}} c_{(a_1, a_2, k_3, \dots, k_n)} = 0.
\]
\end{enumerate}

In particular, if the characteristic of $F$ is\/ $2$, then $f$ is determined by\/ $\oddsupp$ if and only if condition \eqref{cond symmetric} above holds together with
\begin{enumerate}
\item[\rm (B${}_2$)] $c_{(k, k, k_3, \dots, k_n)} = 0$ for all $(k, k, k_3, \dots, k_n) \in N_F^n$ with $k \neq 0$.
\end{enumerate}
\end{proposition}

\begin{proof}
By Fact~\ref{fact oddsupp}, $f$ is determined by $\oddsupp$ if and only if $f$ is totally symmetric (i.e., \eqref{cond symmetric} holds) and $\minor{f}{2}{1}$ does not depend on $x_1$ (i.e., \eqref{cond coeff} holds, by Lemma~\ref{lem:f21}).

Assume then that the characteristic of $F$ is $2$. We need to prove that condition \eqref{cond coeff} is equivalent to (B${}_2$) under the assumption that $f$ is totally symmetric. Let us analyse more carefully the coefficient
\begin{multline*}
d_{(b_1, b_3, \dots, b_n)}
= \sum_{\substack{(a_1, a_2) \in N_F^2 \\ a_1 \oplus_F a_2 = b_1}} c_{(a_1, a_2, b_3, \dots, b_n)} \\
= \underbrace{\sum_{\substack{a_1 \in N_F \phantom{\text{\makebox[0mm]{$)^2$}}} \\ a_1 \oplus_F a_1 = b_1  \phantom{\text{\makebox[0mm]{,}}}}} c_{(a_1, a_1, b_3, \dots, b_n)}}_{\text{(I)}} + \underbrace{\sum_{\substack{(a_1, a_2) \in N_F^2 \\ a_1 < a_2,\, a_1 \oplus_F a_2 = b_1}} (c_{(a_1, a_2, b_3, \dots, b_n)} + c_{(a_2, a_1, b_3, \dots, b_n)})}_{\text{(II)}}.
\end{multline*}
Assuming that $f$ is totally symmetric, we have $c_{(a_1, a_2, b_3, \dots, b_n)} = c_{(a_2, a_1, b_3, \dots, b_n)}$. Hence summand (II) above equals $2 \cdot C$ for some $C \in F$, which is equal to $0$ since $F$ has characteristic $2$.

As for summand (I), observe first that if $F$ is infinite and $b_1$ is odd, then there is no $a_1 \in N_F$ such that $a_1 \oplus_F a_1 = b_1$; hence the sum in (I) is empty and equals $0$. Thus, in this case, we have $d_{(b_1, b_3, \dots, b_n)} = 0$.
Otherwise, i.e., if $F$ is finite or if $F$ is infinite and $b_1$ is even, the sum in (I) has just one summand, namely the one indexed by $a_1 = \tau_F^{-1}(b_1)$ ($\tau_F$ is a bijection by Remark~\ref{rem:tauF}), and we have $d_{(b_1, b_3, \dots, b_n)} = c_{(\tau_F^{-1}(b_1), \tau_F^{-1}(b_1), b_3, \dots, b_n)}$.

By the above observations, we conclude that under the assumption that $F$ has characteristic $2$ and $f$ is totally symmetric, condition \eqref{cond coeff} is equivalent to the condition that $c_{(k, k, k_3, \dots, k_n)} = 0$ for all $(k, k, k_3, \dots, k_n) \in N_F^n$ with $k \neq 0$.
\end{proof}

We reassemble in the following remark some facts that have been established in~\cite{GapA} (more specifically, in the second paragraph of Section~5 and in Theorem~5.2 of~\cite{GapA}).

\begin{remark}
\label{rem:GapA}
Assume that $B$ is a set with a Boolean group structure (i.e., an abelian group such that $x + x = 0$ holds identically). Let $n \geq 3$, and assume that $f \colon A^n \to B$ is a function such that $f|_{A^n_{=}}$ is determined by $\oddsupp$. Fix an element $a \in A$, and let $\varphi \colon A^{n-2} \to B$ be the function given by $\varphi(a_1, \dots, a_{n-2}) := f(a_1, \dots, a_{n-2}, a, a)$ for all $a_1, \dots, a_{n-2} \in A$. (Since $f|_{A^n_{=}}$ is determined by $\oddsupp$, the definition of $\varphi$ is independent from the choice of $a$.) Then $\varphi$ is determined by $\oddsupp$, i.e., $\varphi = \varphi^* \circ \oddsupp|_{A^{n-2}}$ for some function $\varphi^* \colon \mathcal{P}(A) \to B$. Let $\widetilde{\varphi} \colon A^n \to B$ be the function given by
\[
\widetilde{\varphi}(a_1, \dots, a_n) =
\sum_{\substack{k < n \\ 2 | n - k}} \sum_{1 \leq i_1 < \dots < i_k \leq n} \varphi^*(\oddsupp(a_{i_1}, \dots, a_{i_k})),
\]
for all $a_1, \dots, a_n \in A$. Each summand $\varphi^*(\oddsupp(a_{i_1}, \dots, a_{i_k}))$ on the right side is an identification minor of $\varphi$. The function $\widetilde{\varphi}$ is determined by $\oddsupp$ and $\widetilde{\varphi}|_{A^n_{=}} = f|_{A^n_{=}}$.
\end{remark}

\begin{proposition}
\label{prop diagoddsuppchar2}
Let $F$ be a field, and let $f \colon F^n \to F$ be a polynomial function.
If $F$ is finite or the characteristic of $F$ is\/ $2$, then
$f|_{F^n_{=}}$ is determined by\/ $\oddsupp$ if and only if there exist polynomials $P, Q \in F[x_1, \dots, x_n]$ such that $f = \overline{P} + \overline{Q}$, $\overline{P}$ is determined by\/ $\oddsupp$, and $Q$ is a multiple of the polynomial $\Delta_n$.
\end{proposition}

\begin{proof}
For sufficiency, let us assume that
$f = \overline{P} + \overline{Q}$, where $P$ and $Q$ are as in the statement of the proposition. Since $\overline{P}$ is determined by $\oddsupp$, the restriction $\overline{P}|_{F^n_{=}}$ is obviously determined by $\oddsupp$ as well. Moreover, $\overline{Q}|_{F^n_{=}} \equiv 0$ by Lemma~\ref{lemma Delta}. Thus, $f|_{F^n_{=}} = \overline{P}|_{F^n_{=}} + \overline{Q}|_{F^n_{=}} = \overline{P}|_{F^n_{=}}$ is determined by $\oddsupp$.

For necessity, assume first that $F$ is finite. If $f|_{F^n_{=}}$ is determined by $\oddsupp$, then there is a (not necessarily unique) function $g$ such that $g$ is determined by $\oddsupp$ and $f|_{F^n_{=}} = g|_{F^n_{=}}$. By Fact~\ref{fact:field}, $g$ is a polynomial function; hence so is $h = f - g$. By Lemma~\ref{lemma Delta}, $P_h$ is a multiple of the polynomial $\Delta_n$.

Assume then that $F$ is a field of characteristic $2$. Since the additive group of any field of characteristic $2$ is a Boolean group, Remark~\ref{rem:GapA} applies to operations on $F$. Assume that $f \colon F^n \to F$ is a polynomial function such that $f|_{F^n_{=}}$ is determined by $\oddsupp$, and let $\varphi$, $\varphi^*$, and $\widetilde{\varphi}$ be as defined in Remark~\ref{rem:GapA}. Then $\varphi$ is also a polynomial function. The functions $\varphi^*(\oddsupp(a_{i_1}, \dots, a_{i_k}))$, being identification minors of $\varphi$, are polynomial functions. Therefore, Remark~\ref{rem:GapA} implies that $\widetilde{\varphi}$ is a polynomial function and $\widetilde{\varphi}|_{F^n_{=}} = f|_{F^n_{=}}$. Letting $g := \widetilde{\varphi}$ and $h := f - g$, and arguing as in the previous paragraph, we conclude that $P_h$ is a multiple of the polynomial $\Delta_n$.
\end{proof}

\begin{theorem}
\label{thm:char2gapgeq2}
Let $F$ be a field of characteristic $2$, possibly infinite, and let $f \colon F^n \to F$ be a polynomial function of arity at least\/ $4$ which depends on all of its variables. Then\/ $\gap f = p \geq 2$ if and only if there exist polynomials $P, Q \in F[x_1, \dots, x_n]$ such that $f = \overline{P} + \overline{Q}$, $P$ is canonical, $Q$ is a multiple of the polynomial $\Delta_n$, and either
\begin{enumerate}[\rm (a)]
\item \label{item:char2gapgeq2:1}
exactly $n - p$ variables occur in $P$ and $\overline{Q} \neq 0$, or

\item \label{item:char2gapgeq2:2}
$P$ is not a constant polynomial and $\overline{P}$ satisfies conditions \eqref{cond symmetric} and \textnormal{(B${}_2$)} of Proposition~\ref{prop oddsuppchar2}.
\end{enumerate}
Otherwise\/ $\gap f = 1$.
\end{theorem}

\begin{proof}
Combine Theorem~\ref{thm:gap}, Theorem~\ref{thm:sum}, Lemma~\ref{lemma minors}, Lemma~\ref{lemma Delta}, Proposition~\ref{prop oddsuppchar2}, and Proposition~\ref{prop diagoddsuppchar2}, and observe that if $f|_{F^n_{=}}$ is determined by $\oddsupp$ then $\qa f = n$ if and only if $f|_{F^n_{=}}$ is not constant.
\end{proof}

\begin{corollary}
\label{cor decomp char2}
Let $F = \GF(q)$, where $q$ is a power of\/ $2$, and let $f \colon F^n \to F$ be a polynomial function of essential arity $n > \max(q, 3)$. If\/ $\gap f = 2$, then $f$ can be decomposed into a sum of functions of essential arity at most $q - 1$.
\end{corollary}

\begin{proof}
If $n > q$, then $F^n_{=} = F^n$; hence case~\eqref{item:char2gapgeq2:1} in Theorem~\ref{thm:char2gapgeq2} cannot occur, while in case~\eqref{item:char2gapgeq2:2} we have $\overline{Q} \equiv 0$; thus $f = \overline{P}$. Moreover, in case~\eqref{item:char2gapgeq2:2}, every monomial of $P$ involves at most $q - 1$ variables, by condition \textnormal{(B${}_2$)} of Proposition~\ref{prop oddsuppchar2}.
This implies that $f$ can be written as a sum of
functions of essential arity at most $q - 1$,
namely the polynomial functions corresponding to the monomials of $f$.
\end{proof}

\begin{remark}
The decomposition given in Theorem~\ref{thm:char2gapgeq2} is unique only in case~\eqref{item:char2gapgeq2:1}. In case~\eqref{item:char2gapgeq2:2} it can be made unique by requiring that $\overline{P}$ is constant $0$ on $F^n \setminus F^n_{=}$.
\end{remark}

\begin{remark}
Applying Corollary~\ref{cor decomp char2} in the case $q = 2$, we see that any function $f \colon \{0, 1\}^n \to \{0, 1\}$ with essential arity $n \geq 4$ and $\gap f = 2$ can be written as a sum of at most unary functions, i.e., that $f$ is a linear function (cf.\ Example~\ref{ex:linearBf} and \cite{CL2007}).
\end{remark}

\begin{remark}
From the results of \cite{GapA} it follows that if $A$ is a finite set and $B$ is a Boolean group, then every function $f \colon A^n \to B$ with essential arity $n > \max(\card{A}, 3)$ and $\gap f = 2$ can be decomposed into a sum of functions of essential arity at most $n - 2$ (cf.\ Remark~\ref{rem:GapA}).
Corollary~\ref{cor decomp char2} shows that the bound $n - 2$ on the essential arity of the summands can be improved to $q - 1$ (which is independent of $n$) if $A = B = \GF(q)$, where $q$ is a power of $2$ (for further results in this direction see also~\cite{oddsupp}). In the example below, we will construct a polynomial function $f \colon F^n \to F$ over $F = \GF(q)$ for any odd prime power $q$ and any $n \geq 2$, such that $\gap f = 2$ but $f$ cannot be written as a sum of $(n - 1)$-ary functions. This shows that Corollary~\ref{cor decomp char2} does not hold for finite fields with odd characteristic and that the condition of $B$'s being a Boolean group cannot be dropped in the aforementioned result of \cite{GapA}.
\end{remark}

\begin{example}
\label{ex odd}
Let $q$ be an odd prime power, and let $f$ be the polynomial function
\begin{equation}
\label{eq:expansion}
f(x_1, \dots, x_n) = \prod_{i = 1}^n \Bigl( x_i^{q - 1} - \frac{1}{2} \Bigr)
\end{equation}
over $\GF(q)$, where $\frac{1}{2}$ stands for the multiplicative inverse of $2 = 1 + 1$ (it exists, since $\GF(q)$ is of odd characteristic). Let us identify the first two variables of $f$:
\begin{align*}
f(x_1, x_1, x_3, \dots, x_n)
&= \Bigl( x_1^{q - 1} - \frac{1}{2} \Bigr)^2 \cdot \prod_{i = 3}^n \Bigl( x_i^{q - 1} - \frac{1}{2} \Bigr) \\
&= \Bigl( x_1^{2q - 2} - x_1^{q - 1} + \frac{1}{4} \Bigr) \cdot \prod_{i = 3}^n \Bigl( x_i^{q - 1} - \frac{1}{2} \Bigr) \\
&= \frac{1}{4} \cdot \prod_{i = 3}^n \Bigl( x_i^{q - 1} - \frac{1}{2} \Bigr),
\end{align*}
since $x_1^q = x_1$ holds identically in $\GF(q)$. We see that $x_1$ becomes an inessential variable, and $\ess \minor{f}{2}{1} = n - 2$. This together with the total symmetry of $f$ shows that $\gap f = 2$.

Suppose that $f$ is a sum of functions of arity at most $n - 1$.
By Fact~\ref{fact:field}, these functions are polynomial.
This implies that every monomial of $P_f$ involves at most $n - 1$ variables. However, this is clearly not possible, as the expansion of the right side of \eqref{eq:expansion} is a canonical polynomial that involves the monomial $x_1^{q - 1} \cdots x_n^{q - 1}$, which will not be cancelled by any other monomial. This contradiction shows that $f$ cannot be expressed as a sum of functions of arity at most $n - 1$.
\end{example}

%%%%%%%%%%%%%%%%%%%%

\section{Arity gap of polynomial functions over fields of characteristic $0$}
\label{sect char0}

We now consider the case of polynomial functions over fields of characteristic $0$. Unlike polynomial functions over fields of characteristic $2$ (see Proposition~\ref{prop diagoddsuppchar2}), it turns out that in the current case there is no polynomial function $f \colon F^n \to F$ whose restriction $f|_{F^n_{=}}$ is nonconstant and determined by $\oddsupp$.

We first recall the notion of partial derivative in the case of polynomial functions.
We denote the \emph{partial derivative} of a polynomial $p \in F[x_1, \dots, x_n]$ with respect to its $i$-th variable by $\partial_i p$, and we define it by the following rules. The $i$-th partial derivative of a monomial is defined by the rule
\begin{equation}
\partial_i c x_1^{a_1} \cdots x_n^{a_n} =
\begin{cases}
c a_i x_1^{a_1} \cdots x_{i-1}^{a_{i-1}} x_i^{a_i - 1} x_{i+1}^{a_{i+1}} \cdots x_n^{a_n}, & \text{if $a_i \neq 0$,} \\
0, & \text{otherwise.} \\
\end{cases}
\label{eq:dermon}
\end{equation}
Moreover, partial derivatives are additive, i.e.,
\begin{equation}
\partial_i \sum_{j \in J} f_j = \sum_{j \in J} \partial_i f_j.
\label{eq:deradd}
\end{equation}
The partial derivatives of arbitrary polynomials can then be determined by application of \eqref{eq:dermon} and \eqref{eq:deradd}.
The partial derivative of a polynomial function $f \colon F^n \to F$ with respect to its $i$-th variable is denoted by $\partial_i f$, and it is given by $\partial_i f := \overline{\partial_i P_f}$.

Observe that for fields of characteristic $0$, $\partial_i f = 0$ if and only if the $i$-th variable is inessential in $f$.
Also, let us note the difference between
\[
\partial_1 f(x_1, x_1, x_2) = \partial_1 (f(x_1, x_1, x_2))
\quad \text{and} \quad
(\partial_1 f)(x_1, x_1, x_2),
\]
where $f \colon F^3 \to F$ is a polynomial function.
The first one is a partial derivative of an identification minor of $f$, while the second one is an identification minor of a partial derivative of $f$. The chain rule gives the following relationship between these polynomials functions:
\[
\partial_1 f(x_1, x_1, x_2) = (\partial_1 f)(x_1, x_1, x_2) + (\partial_2 f)(x_1, x_1, x_2).
\]
Since we will often consider derivatives of simple minors, it is worth formulating a generalization of the above formula.

\begin{fact}
\label{fact derminor}
Let $F$ be a field of characteristic $0$, let $f \colon F^n \to F$ be a polynomial function, let $\sigma \colon [n] \to [m]$, and let $g \in F^m \to F$ be the simple minor of $f$ defined by $g(x_1, \dots, x_m) = f(x_{\sigma(1)}, \dots, x_{\sigma(n)})$. Then the $j$-th partial derivative of $g$ is
\[
\partial_j g = \sum_{\sigma(i) = j} (\partial_i f) (x_{\sigma(1)}, \dots, x_{\sigma(n)}).
\]
\end{fact}

\begin{lemma}
\label{lemma poly oddsupp}
Let $F$ be a field of characteristic\/ $0$ and let $f \colon F^n \to F$ be a polynomial function of arity at least\/ $2$. If $f|_{F^n_{=}}$ is determined by\/ $\oddsupp$, then $f|_{F^n_{=}}$ is constant, i.e., $\qa f = 0$.
\end{lemma}

\begin{proof}
For $n = 2$, the claim is trivial, so we will assume that $n \geq 3$. Let us suppose that $f|_{F^n_{=}}$ is determined by $\oddsupp$. Then $f(x_1, x_1, x_3, \dots, x_n)$ does not depend on $x_1$ by Fact~\ref{fact oddsupp}; hence we have
\[
(\partial_1 f)(x_1, x_1, x_3, \dots, x_n) + (\partial_2 f)(x_1, x_1, x_3, \dots, x_n) = 0
\]
by Fact~\ref{fact derminor}. Let $\vect{u} = (x_1, x_1, x_1, x_4, \dots, x_n) \in F^n$. From the above equality it follows
that
\[
(\partial_1 f)(\vect{u}) + (\partial_2 f)(\vect{u}) = 0,
\]
and a similar argument shows that
\[
(\partial_1 f)(\vect{u}) + (\partial_3 f)(\vect{u}) = 0
\quad \text{and} \quad
(\partial_2 f)(\vect{u}) + (\partial_3 f)(\vect{u}) = 0.
\]
Since the characteristic of $F$ is different from $2$, by adding these three equalities we can conclude that
\[
(\partial_1 f)(\vect{u}) + (\partial_2 f)(\vect{u}) + (\partial_3 f)(\vect{u}) = 0.
\]
However, according to Fact~\ref{fact derminor}, $(\partial_1 f)(\vect{u}) + (\partial_2 f)(\vect{u}) + (\partial_3 f)(\vect{u})$ is nothing else but the derivative of $f(x_1, x_1, x_1, x_4, \dots, x_n)$ with respect to $x_1$. This implies that $f(x_1, x_1, x_1, x_4, \dots, x_n)$ does not depend on $x_1$, i.e.,
\begin{equation}
f(a, a, a, x_4, \dots, x_n) = f(b, b, b, x_4, \dots, x_n)
\label{eq triple}
\end{equation}
for any $a, b, x_4, \dots, x_n \in F$.

Informally, equality~\eqref{eq triple} expresses the fact that whenever the first three entries of an $n$-tuple are the same, then replacing these three entries with another element of $F$, the value of $f$ does not change. (By symmetry, this is certainly true for any three entries, not only the first three.) From the definition of being determined by $\oddsupp$ it follows immediately that we can also change any two identical entries:
\begin{equation}
f( \cdots a \cdots a \cdots ) = f( \cdots b \cdots b \cdots ).
\label{eq double}
\end{equation}

Let $\vect{x} = (x_1, \dots, x_n)$ be any vector in $F^n_{=}$. We may suppose without loss of generality that $x_1 = x_2$. With the help of \eqref{eq triple} and \eqref{eq double} we can replace the entries of $\vect{x}$ in triples and pairs, until all of them are the same:
\begin{align*}
f(\vect{x})
&= f(\underline{x_1, x_1}, x_3, x_4, x_5, x_6, \dots, x_n) \\
&= f(\underline{x_3, x_3, x_3}, x_4, x_5, x_6, \dots, x_n) \\
&= f(\underline{x_4, x_4}, \underline{x_4, x_4}, x_5, x_6, \dots, x_n) \\
&= f(\underline{x_5, x_5, x_5}, \underline{x_5, x_5}, x_6, \dots, x_n) \\
&= f(\underline{x_6, x_6}, \underline{x_6, x_6}, \underline{x_6, x_6}, \dots, x_n)
 = \cdots \\
&= f(x_n, x_n, x_n, x_n, x_n, x_n, \dots, x_n).
\end{align*}
If $n$ is even, then \eqref{eq double} shows that $f(\vect{x}) = f(\vect{0})$:
\[
f(\vect{x}) = f(\underline{x_n, x_n}, \underline{x_n, x_n}, \dots, \underline{x_n, x_n}) = f(0, 0, 0, 0, \dots, 0, 0);
\]
while if $n$ is odd, then we use both \eqref{eq triple} and \eqref{eq double}:
\[
f(\vect{x}) = f(\underline{x_n, x_n, x_n}, \underline{x_n, x_n}, \dots, \underline{x_n, x_n}) = f(0, 0, 0, 0, 0, \dots, 0, 0).
\]
We have shown that $f(\vect{x}) = f(\vect{0})$ for all $\vect{x} \in F^n_{=}$; hence $f|_{F^n_{=}}$ is indeed constant.
\end{proof}

\begin{lemma}
\label{lemma poly majmin}
Let $F$ be a field of characteristic\/ $0$ and let $f \colon F^3 \to F$ be a polynomial function. If\/ $\gap f = 2$, then\/ $\qa f = 1$.
\end{lemma}

\begin{proof}
By case~\eqref{thmitem3} of Theorem~\ref{thm:gap}, there exist a nonconstant map $h \colon A \to B$ and $i_1, i_2, i_3 \in \{0,1\}$ such that
\begin{align*}
f(x_1, x_0, x_0) &= h(x_{i_1}), \\
f(x_0, x_1, x_0) &= h(x_{i_2}), \\
f(x_0, x_0, x_1) &= h(x_{i_3}).
\end{align*}
Up to permutation of variables there are four possibilities for $(i_1, i_2, i_3)$, namely $(1, 1, 1)$, $(0, 0, 0)$, $(1, 1, 0)$ and $(1, 0, 0)$. We will show that the first three cases cannot occur.

If $(i_1, i_2, i_3) = (1, 1, 1)$ then $f|_{F^3_{=}}$ is determined by $\oddsupp$, and then Lemma~\ref{lemma poly oddsupp} shows that $h$ is constant, a contradiction.

If $(i_1, i_2, i_3) = (0, 0, 0)$ then $f(x_2, x_1, x_1) = f(x_1, x_2, x_1) = f(x_1, x_1, x_2) = h(x_1)$; hence $f(x_2, x_1, x_1)$ does not depend on $x_2$. By Fact~\ref{fact derminor} this means that $(\partial_1 f)(x_2, x_1, x_1) = 0$, in particular, $(\partial_1 f)(x_1, x_1, x_1) = 0$ for all $x_1 \in F$. Similarly, we have $(\partial_2 f)(x_1, x_1, x_1) = (\partial_3 f)(x_1, x_1, x_1) = 0$. Another application of Fact~\ref{fact derminor} yields
\begin{multline*}
\partial_1 h(x_1)
= \partial_1 f(x_1, x_1, x_1) \\
= (\partial_1 f)(x_1, x_1, x_1) + (\partial_2 f)(x_1, x_1, x_1) + (\partial_3 f)(x_1, x_1, x_1)
= 0,
\end{multline*}
and this means that $h$ is constant, a contradiction.

If $(i_1, i_2, i_3) = (1, 1, 0)$, then $f(x_1, x_2, x_2) = f(x_2, x_1, x_2) = f(x_1, x_1, x_2) = h(x_1)$, which does not depend on $x_2$. Again, by Fact~\ref{fact derminor} we see that
\begin{align*}
(\partial_2 f)(x_1, x_2, x_2) + (\partial_3 f)(x_1, x_2, x_2) &= 0, \\
(\partial_1 f)(x_2, x_1, x_2) + (\partial_3 f)(x_2, x_1, x_2) &= 0, \\
(\partial_3 f)(x_1, x_1, x_2) &= 0.
\end{align*}
From these equalities it follows that
\[
(\partial_1 f)(x_1, x_1, x_1) = (\partial_2 f)(x_1, x_1, x_1) = (\partial_3 f)(x_1, x_1, x_1) = 0,
\]
which is again a contradiction.

We are left with the case that $(i_1, i_2, i_3) = (1, 0, 0)$ (up to permutation). This implies that $f|_{F^3_{=}} = h(x_1)|_{F^3_{=}}$, i.e., $\qa f = 1$.
\end{proof}

\begin{theorem}
Let $F$ be a field of characteristic\/ $0$, let $n \geq 2$, and let $P \in F[x_1, \dots, x_n]$ be a polynomial such that all $n$ variables occur in $P$. Then\/ $\gap \overline{P} = p \geq 2$ if and only if there exist polynomials $Q, R \in F[x_1, \dots, x_n]$ such that $P = Q + R$, exactly $n - p$ variables occur in $Q$, and $R$ is a nonzero multiple of the polynomial $\Delta_n$. Otherwise\/ $\gap \overline{P} = 1$. Moreover, the decomposition $P = Q + R$ is unique.
\end{theorem}

\begin{proof}
For necessity, assume that $\gap \overline{P} = p \geq 2$. By Lemma~\ref{lemma poly oddsupp}, if $\overline{P}|_{F^n_{=}}$ is determined by $\oddsupp$, then $\qa \overline{P} = 0$. Theorem~\ref{thm:gap} and Lemma~\ref{lemma poly majmin} then imply that if $\gap \overline{P} = p \geq 2$, then $\qa \overline{P} = n - p$. By Theorem~\ref{thm:sum}, there exist unique functions $g, h \colon A^n \to B$ such that $\overline{P} = g + h$, $h|_{F^n_{=}} \equiv 0$, $h \not\equiv 0$ and $\ess g = n - p$. By Lemma~\ref{lemma minors}, $g$ and $h$ are polynomial functions. Since $F$ is infinite, each one of $g$ and $h$ is induced by a unique polynomial over $F$, namely $P_g$ and $P_h$, respectively. Thus, $P = P_g + P_h$. By Fact~\ref{fact:essential}, exactly $n - p$ variables occur in $P_g$, and by Lemma~\ref{lemma Delta}, $P_h$ is a nonzero multiple of $\Delta(x_1, \dots, x_n)$.

For sufficiency, assume that $P = Q + R$, where $Q$ and $R$ are as in the statement of the theorem. Then $\ess \overline{Q} = n - p$ by Fact~\ref{fact:essential}, and $\overline{R} \not\equiv 0$ and $\overline{R}|_{F^n_{=}} \equiv 0$ by Lemma~\ref{lemma Delta}. From Theorem~\ref{thm:sum} it follows that $\qa \overline{P} = n - p$, and then Theorem~\ref{thm:gap} implies that $\gap \overline{P} = p$.

The uniqueness of the decomposition $P = Q + R$ follows from Theorem~\ref{thm:sum} and from the fact that polynomials and polynomial functions over infinite fields are in one-to-one correspondence.
\end{proof}

Let us note that in the proof of the above theorem we did not really make use of the fact that the function $\overline{P}$ is polynomial; we only used the basic properties of the derivative. Therefore the theorem remains valid for differentiable real functions.

\begin{theorem}
Let $f \colon \mathbb{R}^n \to \mathbb{R}$ be a differentiable function of arity at least\/ $2$. Then\/ $\gap f = p \geq 2$ if and only if there exist differentiable functions $g, h \colon \mathbb{R}^n \to \mathbb{R}$ such that $f = g + h$, $h|_{\mathbb{R}^n_{=}} \equiv 0$, $h \not\equiv 0$, and\/ $\ess g = n - p$. Otherwise\/ $\gap f = 1$. Moreover, the decomposition $f = g + h$ is unique.
\end{theorem}

%%%%%%%%%%%%%%%%%%%%

\section{Some remarks on polynomial functions over infinite fields of odd characteristic}
\label{sec:oddchar}

As the following example illustrates, Proposition~\ref{prop diagoddsuppchar2} and Lemma~\ref{lemma poly oddsupp} do not extend to infinite fields of odd characteristic.

\begin{example}
Let $F$ be an arbitrary field of characteristic $3$, and let $f \colon F^3 \to F$ be the polynomial function induced by
\begin{equation}
\label{eq:last}
2 x^3 + 2 y^3 + 2 z^3 + y z^2 - x y^2 - x z^2 + y^2 z + 2 xyz.
\end{equation}
It is straightforward to verify that
\[
f(x, x, y) = f(x, y, x) = f(y, x, x) = 2 y^3.
\]
Hence $f|_{F^3_{=}}$ is determined by $\oddsupp$ but $f|_{F^3_{=}}$ is not constant. This shows that Lemma~\ref{lemma poly oddsupp} does not hold if $F$ has characteristic $3$.

Next we show that Proposition~\ref{prop diagoddsuppchar2} does not hold for infinite fields of characteristic $3$. Assume now that $F$ is infinite, and let $f$ be induced by \eqref{eq:last}.
Suppose that $g \colon F^3 \to F$ is a polynomial function determined by $\oddsupp$ induced by the canonical polynomial
\[
\sum_{(k_1, k_2, k_3) \in \mathbb{N}^3} c_{(k_1, k_2, k_3)} x_1^{k_1} x_2^{k_2} x_3^{k_3}.
\]
Condition~\eqref{cond coeff} of Proposition~\ref{prop oddsuppchar2} yields the following equalities:
\begin{align*}
& c_{(3, 0, 0)} + c_{(2, 1, 0)} + c_{(1, 2, 0)} + c_{(0, 3, 0)} = 0, \\
& c_{(2, 0, 1)} + c_{(1, 1, 1)} + c_{(0, 2, 1)} = 0, \\
& c_{(1, 0, 2)} + c_{(0, 1, 2)} = 0.
\end{align*}
Taking into account the total symmetry of $g$ (condition \eqref{cond symmetric}) and the fact that the characteristic of $F$ is not $2$, the only solution to this system of equations is $c_{(k_1, k_2, k_3)} = 0$ for all $(k_1, k_2, k_3) \in \mathbb{N}^3$ such that $k_1 + k_2 + k_3 = 3$. Thus, the canonical polynomial of $g(x, x, x)$ does not contain any cubic term; therefore it cannot coincide with $f(x, x, x) = 2 x^3$, and we conclude that $f|_{F^3_{=}} \neq g|_{F^3_{=}}$.
\end{example}

%%%%%%%%%%%%%%%%%%%%

\section*{Acknowledgments}

The first named author is supported by the internal research project F1R-MTH-PUL-12RDO2 of the University of Luxembourg.

The third named author acknowledges that the present project is supported by the Hungarian National Foundation for Scientific Research under grants no.\ K77409 and K83219, by the National Research Fund of Luxembourg, and cofunded under the Marie Curie Actions of the European Commission \hbox{(FP7-COFUND).}

%%%%%%%%%%%%%%%%%%%%


\begin{thebibliography}{99}                                                                                                %
\bibitem{BK}
     \textsc{J. Berman,}  \textsc{A. Kisielewicz,}
     On the number of operations in a clone,
     \textit{Proc.\ Amer.\ Math.\ Soc.} \textbf{122} (1994) 359--369.

\bibitem{CL2007}
     \textsc{M. Couceiro,}  \textsc{E. Lehtonen,}
     On the effect of variable identification on the essential arity of functions on finite sets,
     \textit{Int.\ J. Found.\ Comput.\ Sci.}\ \textbf{18} (2007) 975--986.

\bibitem{CL}
     \textsc{M. Couceiro,}  \textsc{E. Lehtonen,}
     Generalizations of \'{S}wierczkowski's lemma and the arity gap of finite functions,
     \textit{Discrete Math.}\ \textbf{309} (2009) 5905--5912.

\bibitem{CLISMVL2010}
     \textsc{M. Couceiro,}  \textsc{E. Lehtonen,}
     The arity gap of polynomial functions over bounded distributive lattices,
     \textit{40th IEEE International Symposium on Multiple-Valued Logic (ISMVL 2010),}
     IEEE Computer Society, Los Alamitos, 2010, pp.\ 113--116.

\bibitem{GapA}
     \textsc{M. Couceiro,}  \textsc{E. Lehtonen,}  \textsc{T. Waldhauser,}
     Decompositions of functions based on arity gap,
     \textit{Discrete Math.}\ \textbf{312} (2012) 238--247.

\bibitem{GapB}
     \textsc{M. Couceiro,}  \textsc{E. Lehtonen,}  \textsc{T. Waldhauser,}
     The arity gap of order-preserving functions and extensions of pseudo-Boolean functions,
     \textit{Discrete Appl.\ Math.}\ \textbf{160} (2012) 383--390.

\bibitem{oddsupp}
     \textsc{M. Couceiro,}  \textsc{E. Lehtonen,}  \textsc{T. Waldhauser,}
     Additive decomposability of functions over abelian groups,
     arXiv:1105.3464.

\bibitem{Salomaa}
     \textsc{A. Salomaa,}
     On essential variables of functions, especially in the algebra of logic,
     \textit{Ann.\ Acad.\ Sci.\ Fenn.\ Ser.\ A I.\ Math.} \textbf{339} (1963) 3--11.

\bibitem{SK}
     \textsc{S. Shtrakov,}  \textsc{J. Koppitz,}
     On finite functions with non-trivial arity gap,
     \textit{Discuss. Math. Gen. Algebra Appl.}\ \textbf{30} (2010) 217--245.

\bibitem{Willard}
     \textsc{R. Willard,}
     Essential arities of term operations in finite algebras,
     \textit{Discrete Math.} \textbf{149} (1996) 239--259.

\end{thebibliography}
\end{document}